\documentclass[a4paper,twoside,fleqn]{article}
\usepackage{amsfonts}
\usepackage[dvips]{graphics}
\usepackage{amssymb}
\usepackage[reqno]{amsmath}
\usepackage{amsthm}
\usepackage{makeidx}
\usepackage{color}
\usepackage{graphicx}
\usepackage{subfigure}
\usepackage{multicol}
\usepackage{fancyhdr}
\usepackage[colorlinks,linkcolor=blue,anchorcolor=blue,citecolor=blue]{hyperref}

 \newtheorem{theorem}{\sc\bf Theorem}[section]
 
 \newtheorem{corollary}[theorem]{\sc\bf Corollary}
 \newtheorem{lemma}[theorem]{\sc\bf Lemma}
 
 \newtheorem{definition}[theorem]{\sc\bf Definition}
 \newtheorem{remark}[theorem]{\sc\bf Remark}
 \newtheorem{question}[theorem]{\sc\bf Question}
 \newtheorem{example}[theorem]{\sc\bf Example}
  
  \numberwithin{equation}{section}

\makeatletter
\def\@cite#1#2{#1\if@tempswa , #2\fi}
\makeatother

\large

\title{{\bf Spectra originated from semi-B-Fredholm theory and commuting perturbations}
\thanks{This work has been supported by National Natural Science Foundation of China (11171066), Specialized Research Fund for the Doctoral
Program of Higher Education (2010350311001, 20113503120003), Natural
Science Foundation of Fujian Province (2009J01005, 2011J05002) and
Foundation of the Education Department of Fujian Province (JB10042,
JA08036).} }

\author{Qingping \textsc{Zeng},\thanks{Email address: zqpping2003@163.com.}
  \ \ Qiaofen \textsc{Jiang} and Huaijie \textsc{Zhong}
\\ \small (School of Mathematics and Computer Science, Fujian Normal
University, Fuzhou 350007, P.R. China) }

\begin{document}
\date{}
\maketitle

\large

\begin{quote}
 {\bf Abstract:} ~In [\cite{Burgos-Kaidi-Mbekhta-Oudghiri}], Burgos, Kaidi, Mbekhta
 and Oudghiri provided an
 affirmative answer to a question of Kaashoek and Lay and proved that
 an operator $F$ is power finite rank if and only
 if $\sigma_{dsc}(T+F) =\sigma_{dsc}(T)$ for every operator $T$ commuting with
 $F$. Later, several authors extended this result to the
 essential descent spectrum, the left Drazin spectrum and the left essentially Drazin spectrum.
 In this paper, using the theory of
operator with eventual topological uniform descent and the technique
used in [\cite{Burgos-Kaidi-Mbekhta-Oudghiri}], we generalize this
result to various spectra originated from seni-B-Fredholm theory. As
immediate consequences, we give affirmative answers to several
questions posed by Berkani, Amouch and Zariouh. Besides, we provide
a general framework which allows us to derive in a unify way
commuting perturbational results of Weyl-Browder type theorems and
properties (generalized or not). These commuting perturbational
results, in particular, improve many recent results of
[\cite{Berkani-Amouch}, \cite{Berkani-Zariouh partial},
\cite{Berkani Zariouh}, \cite{Berkani Zariouh Functional Analysis},
\cite{Rashid gw}] by removing certain extra assumptions.
    \\
{\bf  2010 Mathematics Subject Classification:} primary 47A10, 47A11; secondary 47A53, 47A55   \\
{\bf Key words:} Semi-B-Fredholm operators; eventual topological
uniform descent; power finite rank; commuting perturbation.
\end{quote}

\section{Introduction }

 \quad\,~In 1972, Kaashoek and Lay have shown in [\cite{Kaashoek-Lay}] that the descent spectrum
is invariant under commuting  power finite rank perturbation $F$
(that is, $F^{n}$ is finite rank for some $n \in \mathbb{N}$). Also
they have conjectured that this perturbation property characterizes
such operators $F$. In 2006, Burgos, Kaidi, Mbekhta and Oudghiri
 provided in [\cite{Burgos-Kaidi-Mbekhta-Oudghiri}] an
 affirmative answer to this question and proved that
 an operator $F$ is power finite rank if and only
 if $\sigma_{dsc}(T+F) =\sigma_{dsc}(T)$ for every operator $T$ commuting with
 $F$. Later, Fredj generalized this result in [\cite{Bel}] to the
 essential descent spectrum. Fredj, Burgos and Oudghiri extended this result in [\cite{Bel-Burgos-Oudghiri 3}] to
the left Drazin spectrum and the left essentially Drazin spectrum.
The present paper is concern with commuting power finite rank
perturbations of semi-B-Fredholm operators. As seen in Theorem
\ref{2.19} (i.e., main result), we generalize the previous results
to various spectra originated from semi-B-Fredholm theory. The proof
of our main result is mainly dependent upon the theory of operator
with eventual topological uniform descent and the technique used in
[\cite{Burgos-Kaidi-Mbekhta-Oudghiri}].

Spectra originated from semi-B-Fredholm theory include, in
particular, the upper semi-B-Weyl spectrum $\sigma_{USBW}$ (resp.
the B-Weyl spectrum $\sigma_{BW}$) which is closely related to
generalized a-Weyl's theorem, generalized a-Browder's theorem,
property $(gw)$ and property $(gb)$ (resp. generalized Weyl's
theorem, generalized Browder's theorem, property $(gaw)$ and
property $(gab)$). Concerning the upper semi-B-Weyl spectrum
$\sigma_{USBW}$, Berkani and Amouch posed in [\cite{Berkani-Amouch}]
the following question:

\begin {question}${\label{1.1}}$
 Let $T \in \mathcal{B}(X)$ and let $N \in \mathcal{B}(X)$ be a
 nilpotent operator commuting with $T$. Do we always have
$$ \sigma_{USBW}(T+N) = \sigma_{USBW}(T) \  ?$$
\end{question}

   Similarly, for the B-Weyl spectrum $\sigma_{BW}$, Berkani and Zariouh posed in [\cite{Berkani
Zariouh}] the following question:

\begin {question}${\label{1.2}}$
 Let $T \in \mathcal{B}(X)$ and let $N \in \mathcal{B}(X)$ be a
 nilpotent operator commuting with $T$. Do we always have
$$ \sigma_{BW}(T+N) = \sigma_{BW}(T) \  ?$$
\end{question}

Recently, Amouch, Zguitti, Berkani and Zariouh have given partial
answers in [\cite{Amouch partial}, \cite{AmouchM ZguittiH},
\cite{Berkani-Amouch}, \cite{Berkani-Zariouh partial}] to Question
\ref{1.1}. As immediate consequences of our main result (see Theorem
\ref{2.19}), we provide positive answers to Questions \ref{1.1} and
\ref{1.2} and some other questions posed by Berkani and Zariouh (see
Corollaries \ref{3.2}, \ref{3.4} and \ref{3.9}). Besides, we provide
a general framework which allows us to derive in a unify way
commuting perturbational results of Weyl-Browder type theorems and
properties (generalized or not). These commuting perturbational
results, in particular, improve many recent results of
[\cite{Berkani-Amouch}, \cite{Berkani-Zariouh partial},
\cite{Berkani Zariouh}, \cite{Berkani Zariouh Functional Analysis},
\cite{Rashid gw}] by removing certain extra assumptions (see
Corollary \ref{3.10} and Remark \ref{3.11}).

\bigskip

 Throughout this paper, let $\mathcal{B}(X)$ denote the Banach
algebra of all bounded linear operators acting on an infinite
dimensional
 complex Banach space $X$, and $\mathcal{F}(X)$ denote its ideal of finite rank operators on $X$.
For an operator $T \in \mathcal{B}(X)$, let $T^{*}$ denote its dual,
$\mathcal {N}(T)$ its kernel, $\alpha(T)$ its nullity, $\mathcal
{R}(T)$ its range, $\beta(T)$ its defect, $\sigma(T)$ its spectrum
and $\sigma_{a}(T)$ its approximate point spectrum. If the range
$\mathcal {R}(T)$ is closed and $\alpha(T) < \infty$ (resp.
$\beta(T) < \infty$), then $T$ is said to be $upper$
$semi$-$Fredholm$ (resp. $lower$ $semi$-$Fredholm$). If $T \in
\mathcal{B}(X)$ is both upper and lower semi-Fredholm, then $T$ is
said to be $Fredholm$. If $T \in \mathcal{B}(X)$ is either upper or
lower semi-Fredholm, then $T$ is said to be $semi$-$Fredholm$, and
its index is defined by
\begin{upshape}ind\end{upshape}$(T)$ = $\alpha(T)-\beta(T)$.

For each $n \in \mathbb{N}$, we set $c_{n}(T) = \dim \mathcal
{R}(T^{n})/\mathcal {R}(T^{n+1})$ and $c^{'}_{n}(T) = \dim \mathcal
{N}(T^{n+1})/\mathcal {N}(T^{n}).$ It follows from
 [\cite{Kaashoek M A 12}, Lemmas 3.1 and 3.2] that, for every $n \in \mathbb{N}$,
$$c_{n}(T) = \dim X / (\mathcal {R}(T) + \mathcal {N}(T^{n})), \ \ \ \  c^{'}_{n}(T) =  \dim \mathcal {N}(T)
\cap \mathcal {R}(T^{n}).$$ Hence, it is easy to see that the
sequences $\{c_{n}(T)\}_{n=0}^{\infty}$ and
$\{c^{'}_{n}(T)\}_{n=0}^{\infty}$ are decreasing. Recall that the
$descent$ and the $ascent$ of $T \in \mathcal{B}(X)$ are
 $dsc(T)= \inf \{n \in \mathbf{\mathbb{N}}:\mathcal {R}(T^{n})= \mathcal {R}(T^{n+1})\}$
 and $asc(T)=\inf \{n \in \mathbf{\mathbb{N}}:\mathcal {N}(T^{n})= \mathcal {N}(T^{n+1})\}$,
 respectively (the infimum of an empty set is defined to be $\infty$). That is,
 $$dsc(T)=\inf \{n \in \mathbf{\mathbb{N}}:c_{n}(T) = 0 \}$$
 and
 $$asc(T)=\inf \{n \in \mathbf{\mathbb{N}}:c^{'}_{n}(T) =0 \}.$$
 Similarly, the $esential$ $descent$ and the $esential$ $ascent$ of $T \in
\mathcal{B}(X)$ are
$$dsc_{e}(T)= \inf \{n \in \mathbf{\mathbb{N}}:c_{n}(T) <
\infty \}$$
 and $$asc_{e}(T)=\inf \{n \in
\mathbf{\mathbb{N}}:c^{'}_{n}(T) < \infty \}.$$  If $asc(T) <
\infty$ and $\mathcal {R}(T^{asc(T)+1})$ is closed, then $T$ is said
to be $left$ $Drazin$ $invertible$. If $dsc(T) < \infty $ and
$\mathcal {R}(T^{dsc(T)})$ is closed, then $T$ is said to be $right$
$Drazin$ $invertible$. If $asc(T) = dsc(T) < \infty $, then $T$ is
said to be $Drazin$ $invertible$. Clearly, $T \in \mathcal{B}(X)$ is
both left and right Drazin invertible if and only if $T$ is Drazin
invertible. If $asc_{e}(T) < \infty$ and $\mathcal
{R}(T^{asc_{e}(T)+1})$ is closed, then $T$ is said to be $left$
$essentially$ $Drazin$ $invertible$. If $dsc_{e}(T) < \infty $ and
$\mathcal {R}(T^{dsc_{e}(T)})$ is closed, then $T$ is said to be
$right$ $essentially$ $Drazin$ $invertible$.

For $T \in \mathcal{B}(X)$, let us define the $left$ $Drazin$
$spectrum$, the $right$ $Drazin$ $spectrum$, the $Drazin$
$spectrum$, the $left$ $essentially$ $Drazin$ $spectrum$, and the
$right$ $essentially$ $Drazin$ $spectrum$ of $T$ as follows
respectively:
$$ \sigma_{LD}(T) = \{ \lambda \in \mathbb{C}: T - \lambda I \makebox{ is not a left Drazin invertible operator} \};$$
$$ \sigma_{RD}(T) = \{ \lambda \in \mathbb{C}: T - \lambda I \makebox{ is not a right Drazin invertible operator} \};$$
$$ \sigma_{D}(T) = \{ \lambda \in \mathbb{C}: T - \lambda I \makebox{ is not a Drazin invertible operator} \};$$
$$ \sigma_{LD}^{e}(T) = \{ \lambda \in \mathbb{C}: T - \lambda I \makebox{ is not a left essentially Drazin invertible operator} \};$$
$$ \sigma_{RD}^{e}(T) = \{ \lambda \in \mathbb{C}: T - \lambda I \makebox{ is not a right essentially Drazin invertible operator} \}.$$
These spectra have been extensively studied by several authors, see
e.g [\cite{Aiena-Biondi-Carpintero}, \cite{AmouchM ZguittiH},
\cite{Berkani}, \cite{Berkani2},
\cite{Carpintero-Garcia-Rosas-Sanabria}, \cite{Bel},
\cite{Bel-Burgos-Oudghiri 3}, \cite{Mbekhta-Muller 9}].

Recall that an operator $T \in \mathcal{B}(X)$ is said to be
$Browder$ (resp. $upper$ $semi$-$Browder$, $lower$ $semi$-$Browder$)
if $T$ is Fredholm and $asc(T)=dsc(T) < \infty$ (resp. $T$ is upper
semi-Fredholm and $asc(T) < \infty$, $T$ is lower semi-Fredholm and
$dsc(T) < \infty$).

For each integer $n$, define $T_{n}$ to be the restriction of $T$ to
$\mathcal{R}(T^{n})$ viewed as the map from $\mathcal{R}(T^{n})$
into $\mathcal{R}(T^{n})$ (in particular $T_{0} = T $). If there
exists $n \in \mathbb{N}$ such that $\mathcal {R}(T^{n})$ is closed
and $T_{n}$ is Fredholm (resp. upper semi-Fredholm, lower
semi-Fredholm, Browder, upper semi-Browder, lower semi-Browder),
then $T$ is called $B$-$Fredholm$ (resp. $upper$
$semi$-$B$-$Fredholm$, $lower$ $semi$-$B$-$Fredholm$, $B$-$Browder$,
$upper$ $semi$-$B$-$Browder$, $lower$ $semi$-$B$-$Browder$). If $T
\in \mathcal {B}(X)$ is upper or lower semi-B-Browder, then $T$ is
called $semi$-$B$-$Browder$. If $T \in \mathcal {B}(X)$ is upper or
lower semi-B-Fredholm, then $T$ is called $semi$-$B$-$Fredholm$. It
follows from [\cite{Berkani semi-B-fredholm}, Proposition 2.1] that
if there exists $n \in \mathbb{N}$ such that $\mathcal {R}(T^{n})$
is closed and $T_{n}$ is
 semi-Fredholm, then $\mathcal {R}(T^{m})$ is closed, $T_{m}$ is semi-Fredholm and
\begin{upshape}ind\end{upshape}$(T_{m})$ = \begin{upshape}ind\end{upshape}$(T_{n})$ for all $m \geq
n$. This enables us to define the index of a semi-B-Fredholm
operator $T$ as the index of the semi-Fredholm operator $T_{n}$,
where $n$ is an integer satisfying $\mathcal {R}(T^{n})$ is closed
and $T_{n}$ is semi-Fredholm. An operator $T \in \mathcal {B}(X)$ is
called $B$-$Weyl$ (resp. $upper$ $semi$-$B$-$Weyl$, $lower$
$semi$-$B$-$Weyl$) if $T$ is B-Fredholm and
\begin{upshape}ind\end{upshape}$(T)=0$ (resp. $T$ is upper
semi-B-Fredholm and
\begin{upshape}ind\end{upshape}$(T) \leq 0$, $T$ is
lower semi-B-Fredholm and
\begin{upshape}ind\end{upshape}$(T) \geq 0$). If $T
\in \mathcal {B}(X)$ is upper or lower semi-B-Weyl, then $T$ is
called $semi$-$B$-$Weyl$.

For $T \in \mathcal{B}(X)$, let us define the $upper$
$semi$-$B$-$Fredholm$ $spectrum$, the $lower$ $semi$-$B$-$Fredholm$
$spectrum$, the $semi$-$B$-$Fredholm$ $spectrum$, the $B$-$Fredholm$
$spectrum$, the $upper$ $semi$-$B$-$Weyl$ $spectrum$, the $lower$
$semi$-$B$-$Weyl$ $spectrum$, the $semi$-$B$-$Weyl$ $spectrum$, the
$B$-$Weyl$ $spectrum$, the $upper$ $semi$-$B$-$Browder$ $spectrum$,
the $lower$ $semi$-$B$-$Browder$ $spectrum$, the
$semi$-$B$-$Browder$ $spectrum$, and the $B$-$Browder$ $spectrum$ of
$T$ as follows respectively:
$$ \sigma_{USBF}(T) = \{ \lambda \in \mathbb{C}: T - \lambda I \makebox{ is not a upper semi-B-Fredholm
operator} \};$$
$$ \sigma_{LSBF}(T) = \{ \lambda \in \mathbb{C}: T - \lambda I \makebox{ is not a lower semi-B-Fredholm
operator} \};$$
$$ \sigma_{SBF}(T) = \{ \lambda \in \mathbb{C}: T - \lambda I \makebox{ is not a semi-B-Fredholm
operator} \};$$
$$ \sigma_{BF}(T) = \{ \lambda \in \mathbb{C}: T - \lambda I \makebox{ is not a B-Fredholm
operator} \};$$
$$ \sigma_{USBW}(T) = \{ \lambda \in \mathbb{C}: T - \lambda I \makebox{ is not a upper semi-B-Weyl
operator} \};$$
$$ \sigma_{LSBW}(T) = \{ \lambda \in \mathbb{C}: T - \lambda I \makebox{ is not a lower semi-B-Weyl
operator} \};$$
$$ \sigma_{SBW}(T) = \{ \lambda \in \mathbb{C}: T - \lambda I \makebox{ is not a semi-B-Weyl
operator} \};$$
$$ \sigma_{BW}(T) = \{ \lambda \in \mathbb{C}: T - \lambda I \makebox{ is not a B-Weyl
operator} \};$$
$$ \sigma_{USBB}(T) = \{ \lambda \in \mathbb{C}: T - \lambda I \makebox{ is not a upper semi-B-Browder
operator} \};$$
$$ \sigma_{LSBB}(T) = \{ \lambda \in \mathbb{C}: T - \lambda I \makebox{ is not a lower semi-B-Browder
operator} \};$$
$$ \sigma_{SBB}(T) = \{ \lambda \in \mathbb{C}: T - \lambda I \makebox{ is not a semi-B-Browder
operator} \};$$
$$ \sigma_{BB}(T) = \{ \lambda \in \mathbb{C}: T - \lambda I \makebox{ is not a B-Browder
operator} \}.$$ These spectra originated from semi-B-Fredholm theory
also have been extensively studied by several authors,  see e.g
[\cite{Aiena-Biondi-Carpintero}, \cite{AmouchM ZguittiH},
\cite{Berkani}, \cite{Berkani pams}, \cite{Berkani semi-B-fredholm},
\cite{Berkani Zariouh}, \cite{Carpintero-Garcia-Rosas-Sanabria}].

For any $T \in \mathcal{B}(X)$, Berkani have found in
[\cite{Berkani}, Theorem 3.6] the following elegant equalities:
$$ \sigma_{LD}(T) = \sigma_{USBB}(T), \ \ \ \  \sigma_{RD}(T) = \sigma_{LSBB}(T);$$
$$ \sigma_{LD}^{e}(T) = \sigma_{USBF}(T), \ \ \ \  \sigma_{RD}^{e}(T) = \sigma_{LSBF}(T);$$
$$ \sigma_{D}(T) = \sigma_{BB}(T).$$

This paper is organized as follows. In Section 2, by using the
theory of operator with eventual topological uniform descent and the
technique used in [\cite{Burgos-Kaidi-Mbekhta-Oudghiri}], we
characterize power finite rank operators via various spectra
originated from seni-B-Fredholm theory. In Section 3, as some
applications, we provide affirmative answers to some questions of
Berkani, Amouch and Zariouh. Besides, we provide a general framework
which allows us to derive in a unify way commuting perturbational
results of Weyl-Browder type theorems and properties (generalized or
not). These commuting perturbational results, in particular, improve
many recent results of [\cite{Berkani-Amouch}, \cite{Berkani-Zariouh
partial}, \cite{Berkani Zariouh}, \cite{Berkani Zariouh Functional
Analysis}, \cite{Rashid gw}] by removing certain extra assumptions.

\section{Main result}

\quad\,~ We begin with the following lemmas in order to give the
proof of the main result in this paper.

\begin {lemma}${\label{2.7}}$ Let $F \in \mathcal {B}(X)$ with $F^{n} \in \mathcal {F}(X)$ for some $n \in \mathbb{N}$.
If $T \in \mathcal {B}(X)$ is upper semi-B-Fredholm and commutes
with $F$, then $T+F$ is also upper semi-B-Fredholm.
\end{lemma}

\begin{proof} Since $T$ is upper semi-B-Fredholm, by [\cite{Berkani}, Theorem
3.6], $T$ is left essentially Drazin invertible. Hence by
[\cite{Bel-Burgos-Oudghiri 3}, Proposition 3.1], $T+F$ is left
essentially Drazin invertible. By [\cite{Berkani}, Theorem 3.6]
again, $T$ is upper semi-B-Fredholm.
\end{proof}

\begin{lemma}${\label{2.10}}$ Let $F \in \mathcal {B}(X)$ with $F^{n} \in \mathcal {F}(X)$ for some $n \in \mathbb{N}$.
If $T \in \mathcal {B}(X)$ is lower semi-B-Fredholm and commutes
with $F$, then $T+F$ is also lower semi-B-Fredholm.
\end{lemma}

\begin{proof} Since $F^{n} \in \mathcal {F}(X)$ for some $n \in
\mathbb{N}$, $\mathcal{R}(F^{n})$ is a closed and finite-dimensional
 subspace, and hence $\dim \mathcal{R}(F^{*n})=\dim \mathcal
 {N}(F^n)^{\perp} = \dim \mathcal{R}(F^{n})$, thus
 $\mathcal{R}(F^{*n})$ is finite-dimensional, this infers that $F^{*n} \in \mathcal {F}(X^*).$
It is obvious that $T^*$ commutes with $F^*$. Since $T$ is lower
semi-B-Fredholm, by [\cite{Berkani}, Theorem 3.6], $T$ is right
essentially Drazin invertible. Then from the presentation before
Section IV of [\cite{Mbekhta-Muller 9}], it follows that $T^{*}$ is
left essentially Drazin invertible. Hence by
[\cite{Bel-Burgos-Oudghiri 3}, Proposition 3.1],
$(T+F)^{*}=T^{*}+F^{*}$ is left essentially Drazin invertible. From
the presentation before Section IV of [\cite{Mbekhta-Muller 9}]
again, it follows that $T+F$ is right essentially Drazin invertible.
Consequently, by [\cite{Berkani}, Theorem 3.6] again, $T+F$ is lower
semi-B-Fredholm.
\end{proof}

It follows from [\cite{Berkani}, Corollary 3.7 and Theorem 3.6] that
$T$ is B-Fredholm if and only if $T$ is both upper and lower
semi-B-Fredholm.

\begin {corollary}${\label{2.11}}$ Let $T \in \mathcal{B}(X)$ and let $F \in \mathcal {B}(X)$ with $F^{n} \in \mathcal {F}(X)$ for some $n \in
\mathbb{N}$. If $T$ commutes with $F$, then

$(1)$ $\sigma_{USBF}(T+F) = \sigma_{USBF}(T);$

$(2)$ $\sigma_{LSBF}(T+F) = \sigma_{LSBF}(T);$

$(3)$ $\sigma_{SBF}(T+F) = \sigma_{SBF}(T);$

$(4)$ $\sigma_{BF}(T+F) = \sigma_{BF}(T).$
\end{corollary}

\begin{proof} From Lemma \ref{2.7}, the first equation follows
easily. The second equation follows immediately from Lemma
\ref{2.10}. The third equation is true because
$\sigma_{SBF}(T)=\sigma_{USBF}(T) \cap \sigma_{LSBF}(T)$, for every
$T \in \mathcal{B}(X)$.
 The fourth equation is
also true because $\sigma_{BF}(T)=\sigma_{USBF}(T) \cup
\sigma_{LSBF}(T)$, for every $T \in \mathcal{B}(X)$.
\end{proof}

To continue the discussion of this paper, we recall some classical
definitions. Using the isomorphism $X/N(T^{d}) \approx R(T^{d})$ and
    following [\cite{Grabiner 9}], a topology on $R(T^{d})$ is
    defined as follows.

\begin{definition} ${\label{2.12}}$ Let $T \in \mathcal{B}(X)$. For every $d \in \mathbf{\mathbb{N}},$
the operator range topological on $R(T^{d})$ is defined by the norm
$||\small{\cdot}||_{R(T^{d})}$ such that for all $y \in R(T^{d})$,
$$||y||_{R(T^{d})} = \inf\{||x||: x \in X,y=T^{d}x\}.$$
\end{definition}
 For a detailed discussion of operator ranges and their topologies,
 we refer the reader to [\cite{Fillmore-Williams 7}] and [\cite{Grabiner 8}].
If $T\in\mathcal{B}(X)$, for each $n \in \mathbb{N}$, $T$ induces a
linear transformation from the vector space $\mathcal {R}(T^{n})/
\mathcal {R}(T^{n+1})$ to the space $\mathcal {R}(T^{n+1})/\mathcal
{R}(T^{n+2})$. We will let $k_{n}(T)$ be the dimension of the null
space of the induced map. From [\cite{Grabiner 9}, Lemma 2.3] it
follows that, for every $n \in \mathbb{N}$,
\begin{align*} \qquad  \qquad \qquad  k_{n}(T)
 &= \dim (\mathcal {N}(T) \cap \mathcal {R}(T^{n})) / (\mathcal {N}(T) \cap \mathcal {R}(T^{n+1})) \\
 &= \dim (\mathcal {R}(T) +
\mathcal {N}(T^{n+1})) / (\mathcal {R}(T) + \mathcal {N}(T^{n})).
 \end{align*}

\begin {definition}${\label{2.13}}$ Let $T \in
    \mathcal{B}(X)$ and let $d \in \mathbf{\mathbb{N}}$. Then $T$ has
    $uniform$ $descent$ for $n \geq d$ if $k_{n}(T)=0$ for all $n \geq d$. If in addition $R(T^{n})$
    is closed in the operator range topology of $R(T^{d})$ for all $n \geq d$$,$ then we say that $T$
    has $eventual$ $topological$ $uniform$
    $descent$$,$ and$,$ more precisely$,$ that $T$ has $topological$ $uniform$
    $descent$ $for$ $n \geq d$.

\end{definition}

Operators with eventual topological uniform descent are introduced
by Grabiner in [\cite{Grabiner 9}]. It includes all classes of
operators introduced in the Introduction of this paper. It also
includes many other classes of operators such as operators of Kato
type, quasi-Fredholm operators, operators with finite descent and
operators with finite essential descent, and so on. A very detailed
and far-reaching account of these notations can be seen in
[\cite{Aiena}, \cite{Berkani}, \cite{Mbekhta-Muller 9}]. Especially,
operators which have topological uniform descent for $n \geq 0$ are
precisely the $semi$-$regular$ operators studied by Mbekhta in
[\cite{Mbekhta}]. Discussions of operators with eventual topological
uniform descent may be found in [\cite{Berkani-Castro-Djordjevic},
\cite{Cao}, \cite{Grabiner 9}, \cite{Jiang-Zhong-Zeng},
\cite{Jiang-Zhong-Zhang}, \cite{Zeng-Zhong-Wu}].

An operator $T \in \mathcal{B}(X)$ is said to be $essentially$
$semi$-$regular$ if $\mathcal {R}(T)$ is closed and
$k(T):=\sum_{n=0}^{\infty} k_{n}(T) <\infty$. From [\cite{Grabiner
9}, Theorem 3.7] it follows that $$k(T) = \dim \mathcal
{N}(T)/(\mathcal {N}(T) \cap \mathcal {R}(T^{\infty})) = \dim
(\mathcal {R}(T) + \mathcal {N}(T^{\infty}))/\mathcal {R}(T).$$
Hence, every essentially semi-regular operator $T \in
\mathcal{B}(X)$ can be characterized by $\mathcal {R}(T)$ is closed
and and there exists a finite dimensional subspace
 $F \subseteq X$ such that $\mathcal {N}(T) \subseteq \mathcal
{R}(T^{\infty}) + F.$ In addition, if $T$ is essentially
semi-regular, then $T^{n}$ is essentially semi-regular,
 and hence $R(T^{n})$ is closed for all $n \in \mathbf{\mathbb{N}}$ (see Theorem 1.51 of [\cite{Aiena}]).
Hence it is easy to verify that
  if $T \in \mathcal{B}(X)$ is essentially semi-regular, then there exist $p \in \mathbb{N}$ such
that $T$ has topological uniform descent for
 $n \geq p$.

 Also, an operator $T \in \mathcal{B}(X)$ is called Riesz
if its essential spectrum $\sigma_e(T):= \{ \lambda \in \mathbb{C}:
T - \lambda I \makebox{ is not Fredholm} \}=\{0\}$. The $hyperrange$
and $hyperkernel$ of $T \in \mathcal{B}(X)$ are the subspaces of $X$
defined by $\mathcal {R}(T^{\infty}) = \bigcap_{n=1}^{\infty}
\mathcal {R}(T^{n})$ and $\mathcal {N}(T^{\infty}) =
\bigcup_{n=1}^{\infty} \mathcal {N}(T^{n})$, respectively.

\begin{lemma} ${\label{2.14}}$
Suppose that $T \in \mathcal{B}(X)$ has topological uniform descent
for $m \geq d$. If\, $S \in
  \mathcal{B}(X)$ is a Riesz operator commuting with $T$
   and $V=S+T$ has topological uniform descent
for $n \geq l,$ then:

 $(a)\ \mathrm{dim\,} (\mathcal {R}(T^{\infty})+ \mathcal {R}(V^{\infty})) / (\mathcal {R}(T^{\infty}) \cap \mathcal {R}(V^{\infty})) <
\infty;$

 $(b)\ \mathrm{dim\,} (\overline{\mathcal {N}(T^{\infty})} +
\overline{\mathcal {N}(V^{\infty})}) / ( \overline{\mathcal
{N}(T^{\infty})} \cap \overline{\mathcal {N}(V^{\infty})}) <
\infty;$

  $(c)$\,\,$\mathrm{dim\,} \mathcal {R}(V^{n})/\mathcal {R}(V^{n+1})$ $=$
$\mathrm{dim\,} \mathcal {R}(T^{m})/\mathcal {R}(T^{m+1})$ for
sufficiently large $m$ and $n;$

  $(d)$\,\,$\mathrm{dim\,} \mathcal {N}(V^{n+1})/\mathcal {N}(V^{n})$ $=$
$\mathrm{dim\,} \mathcal {N}(T^{m+1})/\mathcal {N}(T^{m})$ for
sufficiently large $m$ and $n$.
\end{lemma}
\begin{proof} Parts (c) and (d) follow directly from
[\cite{Zeng-Zhong-Wu},  Theorems 3.8 and 3.12 and Remark 4.5].

When $d \neq 0$ (that is, $T$ is not semi-regular), parts (a) and
(b) follow also directly from [\cite{Zeng-Zhong-Wu},  Theorems 3.8
and 3.12 and Remark 4.5].

When $d = 0$ (that is, $T$ is semi-regular), then by
[\cite{Zeng-Zhong-Wu},  Theorems 3.8] we have that $V=T+S$ is
essentially semi-regular. So, there exist $p \in \mathbb{N}$ such
that $V$ has topological uniform descent for $n \geq p$. If $p \neq
0$ (that is, $V$ is not semi-regular), then parts (a) and (b) follow
directly from [\cite{Zeng-Zhong-Wu},  Theorem 3.12 and Remark 4.5].
If $p=0$ (that is, $V$ is semi-regular), noting that $(M+N)/N
\approx M/(M \cap N)$ for any subspaces $M$ and $N$ of $X$ (see
[\cite{Kaashoek M A 12}, Lemma 2.2]), then parts (a) and (b) follow
from [\cite{Zeng-Zhong-Wu}, Theorems 3.8 and Remark 4.5].
\end{proof}

\begin {theorem}${\label{2.15}}$ Let $F \in \mathcal {B}(X)$ with $F^{n} \in \mathcal {F}(X)$ for some $n \in \mathbb{N}$.

$(1)$\  If $T \in \mathcal {B}(X)$ is semi-B-Fredholm and commutes
with $F$, then

     $(a)\ \mathrm{dim\,} (\mathcal {R}(T^{\infty})+ \mathcal {R}((T+F)^{\infty})) / (\mathcal {R}(T^{\infty}) \cap \mathcal {R}((T+F)^{\infty})) <
     \infty;$

     $(b)\ \mathrm{dim\,} (\overline{\mathcal {N}(T^{\infty})} +
     \overline{\mathcal {N}((T+F)^{\infty})}) / ( \overline{\mathcal {N}(T^{\infty})} \cap
   \overline{\mathcal {N}((T+F)^{\infty})}) < \infty .$

$(2)$ If $T \in \mathcal {B}(X)$ is upper \begin {upshape}(\end
{upshape}resp. lower\begin {upshape})\end {upshape} semi-B-Fredholm
and commutes with $F$, then $T+F$ is also upper \begin
{upshape}(\end {upshape}resp. lower\begin {upshape})\end {upshape}
semi-B-Fredholm and
           \begin{upshape}ind\end{upshape}$(T+F)=$\begin{upshape}ind\end{upshape}$(T)$.
           \end{theorem}

\begin{proof} Suppose that $F \in \mathcal {B}(X)$ with $F^{n} \in \mathcal {F}(X)$ for some $n \in
\mathbb{N}$. Then $F^{n}$ is Riesz, that is,
$\sigma_e(F^{n})=\{0\}.$ By the spectral mapping theorem for the
essential spectrum, we get that $\sigma_e(F)=\{0\},$ so $F$ is
Riesz.

$(1)$ Since $T$ is semi-B-Fredholm and commutes with $F$, by Lemmas
\ref{2.7} and \ref{2.10}, $T+F$ is also semi-B-Fredholm. Since every
semi-B-Fredholm operator is an operator of eventual topological
uniform descent, by Lemma \ref{2.14}$(a)$ and $(b)$, parts $(a)$ and
$(b)$ follow immediately.

$(2)$ By Lemmas \ref{2.7} and \ref{2.10}, it remains to prove that
\text{ind}$(T+F)=$\text{ind}$(T)$. Since every semi-B-Fredholm
operator is an operator of eventual topological uniform descent, by
Lemma \ref{2.14}$(c)$ and $(d)$ and [\cite{Berkani semi-B-fredholm},
Proposition 2.1], we have that \text{ind}$(T+F) =$\text{ind}$(T)$.
\end{proof}

\begin{theorem}${\label{2.16}}$ Let $T \in \mathcal{B}(X)$ and let $F \in \mathcal {B}(X)$ with $F^{n} \in \mathcal {F}(X)$ for some $n \in
\mathbb{N}$. If $T$ commutes with $F$, then

$(1)$ $\sigma_{USBW}(T+F) = \sigma_{USBW}(T);$

$(2)$ $\sigma_{LSBW}(T+F) = \sigma_{LSBW}(T);$

$(3)$ $\sigma_{SBW}(T+F) = \sigma_{SBW}(T);$

$(4)$ $\sigma_{BW}(T+F) = \sigma_{BW}(T).$

\end{theorem}

\begin{proof}It follows directly from Theorem \ref{2.15}(2).
\end{proof}

Next, we turn to the discussion of characterizations of power finite
rank operators via various spectra originated from seni-B-Fredholm
theory. Before this, some notations are needed.

For $T \in \mathcal{B}(X)$, let us define the $descent$ $spectrum$,
the $essential$ $descent$ $spectrum$ and the $eventual$
$topological$ $uniform$ $descent$ $spectrum$ of $T$ as follows
respectively: $$\sigma_{dsc}(T)=\{\lambda \in \mathbb{C}:
dsc(T-\lambda I) =\infty \};$$ $$\sigma^{e}_{dsc}(T)=\{\lambda \in
\mathbb{C}: dsc_{e}(T-\lambda I) =\infty \};$$
$$\sigma_{ud}(T)=\{\lambda \in \mathbb{C}: T-\lambda I \makebox{
does not have eventual topological uniform descent} \}.$$

 In [\cite{Jiang-Zhong-Zhang}], Jiang, Zhong
and Zhang obtained a classification of the components of $eventual$
$topological$ $uniform$ $descent$ $resolvent$ $set$ $\rho_{ud}(T):=
\mathbb{C} \backslash \sigma_{ud}(T)$.  As an application of the
classification, they show that $\sigma_{ud}(T)=\varnothing$
precisely when $T$ is algebraic.

\begin {lemma}${\label{2.17}}$ \begin {upshape}([\cite{Jiang-Zhong-Zhang}, Corollary 4.5])\end {upshape}
Let $T \in \mathcal{B}(X)$ and let $\sigma_{*} \in \{\sigma_{ud},
\sigma_{dsc}, \sigma^{e}_{dsc}, \sigma_{USBF}=\sigma^{e}_{LD}, $
$\sigma_{USBB}=\sigma_{LD}, \sigma_{BB}=\sigma_{D} \}$. Then the
following statements are equivalent:

 $(1)$ $\sigma_{*}(T) = \varnothing$;

 $(2)$ $T$ is algebraic \begin {upshape}(\end {upshape}that is, there exists a non-zero complex
polynomial $p$ for which $p(T)=0$\begin {upshape})\end {upshape}.
\end{lemma}

\begin {corollary}${\label{2.18}}$ Let $T \in \mathcal{B}(X)$ and let $\sigma_{*} \in \{
\sigma_{ud}, \sigma_{dsc}, \sigma^{e}_{dsc},
\sigma_{USBF}=\sigma^{e}_{LD}, \sigma_{LSBF}= \sigma^{e}_{RD},
\sigma_{SBF},$ $ \sigma_{BF},\sigma_{USBW}, \sigma_{LSBW},
\sigma_{SBW}, \sigma_{BW},\sigma_{USBB}=\sigma_{LD},
\sigma_{LSBB}=\sigma_{RD}, \sigma_{SBB}, \sigma_{BB}=\sigma_{D} \}$.
Then the following statements are equivalent:

 $(1)$ $\sigma_{*}(T) = \varnothing$;

 $(2)$ $T$ is algebraic.
\end{corollary}

\begin{proof} If $\sigma_{*} \in \{
\sigma_{ud}, \sigma_{dsc}, \sigma^{e}_{dsc},
\sigma_{USBF}=\sigma^{e}_{LD}, $ $\sigma_{USBB}=\sigma_{LD},
\sigma_{BB}=\sigma_{D}\}$, the conclusion is given by Lemma
\ref{2.17}. Note that $$\sigma_{ud}(\cdot) \subseteq
\sigma_{SBF}(\cdot) \subseteq
\{_{\sigma_{USBF}(\cdot)=\sigma^{e}_{LD}(\cdot)}^{\sigma_{SBW}(\cdot)}
\subseteq \sigma_{USBW}(\cdot) \subseteq
\{_{\sigma_{USBB}(\cdot)=\sigma_{LD}(\cdot)}^{\sigma_{BW}(\cdot)}
\subseteq
     \sigma_{BB}(\cdot)=\sigma_{D}(\cdot)$$
and that $$\sigma_{ud}(\cdot) \subseteq \sigma_{SBF}(\cdot)
\subseteq
\{_{\sigma_{LSBF}(\cdot)=\sigma^{e}_{RD}(\cdot)}^{\sigma_{SBW}(\cdot)}
\subseteq \sigma_{LSBW}(\cdot) \subseteq
\{_{\sigma_{LSBB}(\cdot)=\sigma_{RD}(\cdot)}^{\sigma_{BW}(\cdot)}
\subseteq
     \sigma_{BB}(\cdot)=\sigma_{D}(\cdot).$$
By Lemma \ref{2.17}, if $\sigma_{*} \in \{ \sigma_{LSBF}=
\sigma^{e}_{RD}, \sigma_{SBF},$ $\sigma_{USBW}, \sigma_{LSBW},
\sigma_{SBW}, \sigma_{BW}, \sigma_{LSBB}=\sigma_{RD} \}$, the
conclusion follows easily. Note that $$\sigma_{ud}(\cdot) \subseteq
\sigma_{SBB}(\cdot) \subseteq
     \sigma_{BB}(\cdot)=\sigma_{D}(\cdot)$$
     and that $$\sigma_{ud}(\cdot) \subseteq
\sigma_{BF}(\cdot) \subseteq
     \sigma_{BB}(\cdot)=\sigma_{D}(\cdot).$$
Again by Lemma \ref{2.17}, if $\sigma_{*} \in \{ \sigma_{SBB},
\sigma_{BF}\}$, the conclusion follows easily.
\end{proof}

In [\cite{Bel-Burgos-Oudghiri 3}, Theorem 3.2], O. Bel Hadj Fredj et
al. proved that $F \in \mathcal {B}(X)$ with $F_{n} \in \mathcal
{F}(X)$ for some $n \in \mathbb{N}$ if and only if
$\sigma_{LD}^{e}(T+F) =\sigma_{LD}^{e}(T)$ (equivalently,
$\sigma_{LD}(T+F) =\sigma_{LD}(T)$) for every operator $T$ in the
commutant of $F$.

\smallskip

We are now in a position to give the proof of the following main
result.

\begin {theorem}${\label{2.19}}$ Let $F \in \mathcal {B}(X)$ and $\sigma_{*} \in \{
\sigma_{dsc}, \sigma^{e}_{dsc}, \sigma_{USBF}=\sigma^{e}_{LD},
\sigma_{LSBF}= \sigma^{e}_{RD}, \sigma_{SBF},$ $
\sigma_{BF},\sigma_{USBW}, \sigma_{LSBW}, \sigma_{SBW},
\sigma_{BW},\sigma_{USBB}=\sigma_{LD}, \sigma_{LSBB}=\sigma_{RD},
\sigma_{SBB}, \sigma_{BB}=\sigma_{D} \}$. Then the following
statements are equivalent:

$(1)$ $F^{n} \in \mathcal {F}(X)$ for some $n \in \mathbb{N}$;

$(2)$ $\sigma_{*}(T+F) =\sigma_{*}(T)$ for all $T \in \mathcal
{B}(X)$ commuting with $F$.
\end{theorem}

\begin{proof} For $\sigma_{*} \in \{\sigma_{dsc}, \sigma^{e}_{dsc},
\sigma_{USBB}=\sigma_{LD}, \sigma_{USBF}=\sigma^{e}_{LD}\}$, the
conclusion can be found in [\cite{Burgos-Kaidi-Mbekhta-Oudghiri},
Theorem 3.1], [\cite{Bel}, Theorem 3.1] and
[\cite{Bel-Burgos-Oudghiri 3}, Theorem 3.2]. In the following, we
prove the conclusion for the others spectra.

$(1)\Rightarrow (2)$ For $\sigma_{*} \in \{\sigma_{LSBF}=
\sigma^{e}_{RD}, \sigma_{SBF},$ $ \sigma_{BF},\sigma_{USBW},
\sigma_{LSBW}, \sigma_{SBW}, \sigma_{BW}\}$, the conclusion follows
directly from Corollary \ref{2.11} and Theorem \ref{2.16}.

For $\sigma_{*} \in \{\sigma_{LSBB}=\sigma_{RD}\}$, suppose that $F
\in \mathcal {B}(X)$ with $F^{n} \in \mathcal {F}(X)$ for some $n
\in \mathbb{N}$ and that $T \in \mathcal {B}(X)$ commutes with $F$.
It is clear that $F^* \in \mathcal {B}(X^*)$ with $F^{*n} \in
\mathcal {F}(X^*)$ and that $T^* \in \mathcal {B}(X^*)$ commutes
with $F^*$. From the presentation before this theorem, we get that
$\sigma_{LD}(T^*+F^*) =\sigma_{LD}(T^*)$, hence dually,
$\sigma_{RD}(T+F) =\sigma_{RD}(T)$.

For $\sigma_{*} \in \{\sigma_{SBB}, \sigma_{BB}=\sigma_{D}\}$,
noting that $\sigma_{SBB}(\cdot) = \sigma_{USBB}(\cdot) \cap
\sigma_{LSBB}(\cdot)$ and that $\sigma_{BB}(\cdot) =
\sigma_{USBB}(\cdot) \cup \sigma_{LSBB}(\cdot)$, the conclusion
follows.

$(2) \Rightarrow (1)$ Conversely, suppose that $\sigma_{*}(T+F)
=\sigma_{*}(T)$ for all $T \in \mathcal {B}(X)$ commuting with $F$,
where $\sigma_{*} \in \{\sigma_{LSBF}= \sigma^{e}_{RD},
\sigma_{SBF},$ $ \sigma_{BF},\sigma_{USBW}, \sigma_{LSBW},
\sigma_{SBW}, \sigma_{BW}, \sigma_{LSBB}=\sigma_{RD}, \sigma_{SBB},
\sigma_{BB}=\sigma_{D} \}$. By considering $T=0$, then
$\sigma_{*}(F) =\sigma_{*}(0+F)=\sigma_{*}(0)=\varnothing.$ By
Corollary \ref{2.18}, we know that $F$ is algebraic. Therefore
  $$X =X_{1} \oplus X_{2}\oplus \cdots \oplus X_{n},$$
 where $\sigma(F)=\{\lambda_{1},\lambda_{2},\cdots ,\lambda_{1} \}$
and the restriction of $F-\lambda_{i}$ to $X_{i}$ is nilpotent for
every $1 \leq i \leq n$. We claim that if $\lambda_{i} \neq 0$,
$\dim X_{i}$ is finite. Suppose to the contrary that $\lambda_{i}
\neq 0$ and $ X_{i}$ is infinite dimensional. For every $1 \leq j
\leq n$, let $F_{j}$ be the restriction of $F$ to $X_{j}$. Then with
respect to the decomposition $X =X_{1} \oplus X_{2}\oplus \cdots
\oplus X_{n}$,
$$F =F_{1} \oplus F_{2}\oplus \cdots \oplus F_{n}.$$ By
[\cite{Burgos-Kaidi-Mbekhta-Oudghiri}, Proposition 3.3], there
exists a non-algebraic operator $S_{i}$ on $X_{i}$ commuting with
the restriction $F_{i}$ of $F$. Let $S$ denote the extension of
$S_{i}$ to $X$ given by $S=0$ on each $X_{j}$ such that $j \neq i$.
Obviously $SF = FS$, and so $\sigma_{*}(S+F) =\sigma_{*}(S)$ by
hypothesis. On the other hand, since $F =F_{1} \oplus F_{2}\oplus
\cdots \oplus F_{n}$ is algebraic, $F_{j}$ is algebraic for every $1
\leq j \leq n$. In particular, $F_{j}$ is algebraic for every $j
\neq i$. Hence by Corollary \ref{2.18},
$\sigma_{BB}(F_{j})=\varnothing$ for every $j \neq i$, it follows
easily that $\sigma_{*}(S+F)= \sigma_{*}(S_{i}+F_{i})$. Since
$\sigma_{*}(S)= \sigma_{*}(S_{i})$, we obtain that
$\sigma_{*}(S_{i}) = \sigma_{*}(S_{i} + F_{i}) = \sigma_{*}(S_{i} +
\lambda_{i})$ because $F_{i} - \lambda_{i}$ is nilpotent. Choose an
arbitrary complex number $\alpha \in \sigma_{*}(S) \neq
\varnothing$, it follows that $k\lambda_{i}+\alpha \in
\sigma_{*}(S)$ for every positive integer $k$ , which implies that
$\lambda_{i} = 0$, the desired contradiction.
\end{proof}

\begin {remark}${\label{2.20}}$ \begin{upshape} $(1)$ The argument we have given for the the
implication $(2) \Rightarrow (1)$ in Theorem \ref{2.19} is, in fact,
discovered by following the trail marked out by Burgos, Kaidi,
Mbekhta and Oudghiri [\cite{Burgos-Kaidi-Mbekhta-Oudghiri}].

$(2)$ By [\cite{Berkani-Amouch}, Lemma 2.3], [\cite{Mbekhta-Muller
9}, pp. 135-136] and a similar argument of [\cite{Berkani pams},
Proposition 3.3], we know that $\sigma_{*}(T+F)=\sigma_{*}(T)$ for
all finite rank operator $F$ not necessarily commuting with $T$,
where $\sigma_{*} \in \{\sigma^{e}_{dsc},
\sigma_{USBF}=\sigma^{e}_{LD}, \sigma_{LSBF}= \sigma^{e}_{RD},
\sigma_{SBF},$ $ \sigma_{BF},\sigma_{USBW}, \sigma_{LSBW},
\sigma_{SBW}, \sigma_{BW}\}$. By [\cite{Mbekhta-Muller 9},
Observation 5 in p. 136], we know that $\sigma_{*}$ is not stable
under non-commuting finite rank perturbation, where $\sigma_{*} \in
\{\sigma_{dsc}, \sigma_{USBB}=\sigma_{LD},
\sigma_{LSBB}=\sigma_{RD}, \sigma_{SBB}, \sigma_{BB}=\sigma_{D}\}$.

\end{upshape}
\end{remark}

\section{Some applications}

 \quad\,~Rashid claimed in [\cite{Rashid}, Theorem 3.15] that if $T \in
\mathcal{B}(X)$ and $Q$ is a quasi-nilpotent operator that commute
with $T$, then (in [\cite{Rashid}], $\sigma_{USBW}$ is denoted as
$\sigma_{SBF^{-}_{+}}$)
$$\sigma_{USBW}(T+Q)=\sigma_{USBW}(T).$$
In [\cite{Zeng-Zhong}, Example 2.13], the authors showed that this
equality does not hold in general.

As an immediate consequence of Theorem \ref{2.19} (that is, main
result), we obtain the following corollary which, in particular, is
a corrected version of [\cite{Rashid}, Theorem 3.15] and also
provide positive answers to Questions \ref{1.1} and \ref{1.2}.

\begin {corollary}${\label{3.2}}$ Let $T \in \mathcal{B}(X)$ and let $N \in \mathcal{B}(X)$ be a
 nilpotent operator commuting with $T$. Then

 $(1)$ $\sigma_{USBW}(T+N) = \sigma_{USBW}(T);$

 $(2)$ $\sigma_{LSBW}(T+N) = \sigma_{LSBW}(T);$

 $(3)$ $\sigma_{SBW}(T+N) = \sigma_{SBW}(T);$

 $(4)$ $\sigma_{BW}(T+N) = \sigma_{BW}(T).$

\end{corollary}

Besides Question \ref{1.2}, Berkani and Zariouh also posed in
[\cite{Berkani Zariouh}] the following  question:

\begin {question}${\label{3.3}}$
 Let $T \in \mathcal{B}(X)$ and let $N \in \mathcal{B}(X)$ be a
 nilpotent operator commuting with $T$. Under which conditions
$$\sigma_{BF}(T+N) = \sigma_{BF}(T) \ ?$$
\end{question}

As an immediate consequence of Theorem \ref{2.19}, we also obtain
the following corollary which, in particular, provide a positive
answer to Question \ref{3.3}.

\begin {corollary}${\label{3.4}}$ Let $T \in \mathcal{B}(X)$ and let $N \in \mathcal{B}(X)$ be a
 nilpotent operator commuting with $T$. Then

 $(1)$ $\sigma_{USBF}(T+N) = \sigma_{USBF}(T);$

 $(2)$ $\sigma_{LSBF}(T+N) = \sigma_{LSBF}(T);$

 $(3)$ $\sigma_{SBF}(T+N) = \sigma_{SBF}(T);$

 $(4)$ $\sigma_{BF}(T+N) = \sigma_{BF}(T).$

\end{corollary}

We say that $\lambda \in \sigma_{a}(T)$ is a left pole of $T$ if
$T-\lambda I$ is left Drazin invertible. Let $\Pi_{a}(T)$ denote the
set of all left poles of $T$. An operator $T \in \mathcal{B}(X)$ is
called $a$-$polaroid$ if $\makebox{iso}\sigma_{a}(T)=\Pi_{a}(T)$.
Here and henceforth, for $A \subseteq \mathbb{C}$, $\makebox{iso}A$
is the set of isolated points of $A$. Besides Questions \ref{1.2}
and \ref{3.3}, Berkani and Zariouh also posed in [\cite{Berkani
Zariouh Functional Analysis}] the following three questions:

\begin {question}${\label{3.5}}$
 Let $T \in \mathcal{B}(X)$ and let $N \in \mathcal{B}(X)$ be a
 nilpotent operator commuting with $T$. Under which conditions
$$asc(T+N) < \infty \Longleftrightarrow \sigma_{asc}(T) < \infty \ ?$$
\end{question}

\begin {question}${\label{3.6}}$
 Let $T \in \mathcal{B}(X)$ and let $N \in \mathcal{B}(X)$ be a
 nilpotent operator commuting with $T$. Under which conditions,
$\mathcal {R}((T+N)^{m})$ is closed for $m$ large enough if and only
if $\mathcal {R}(T^{m})$ is closed for $m$ large enough $?$
\end{question}

\begin {question}${\label{3.7}}$
 Let $T \in \mathcal{B}(X)$ and let $N \in \mathcal{B}(X)$ be a
 nilpotent operator commuting with $T$. Under which conditions
$$ \Pi_{a}(T+N) = \Pi_{a}(T) \ ?$$
\end{question}

We mention that Question \ref{3.5} is, in fact, an immediate
consequence of an earlier result of Kaashoek and Lay
[\cite{Kaashoek-Lay}, Theorem 2.2]. To Question \ref{3.6}, suppose
that $T \in \mathcal{B}(X)$ and that $N \in \mathcal{B}(X)$ is a
 nilpotent operator commuting with $T$. As a
direct consequence of Theorem \ref{2.19} (that is, main result), we
know that if there exists $n \in \mathbb{N}$ such that $c_{n}(T) <
\infty$ or $c^{'}_{n}(T) < \infty$, then $\mathcal {R}((T+N)^{m})$
is closed for $m$ large enough if and only if $\mathcal {R}(T^{m})$
is closed for $m$ large enough.

 To Question \ref{3.7}, we first recall a classical result.

 \begin {lemma}${\label{3.8}}$ \begin{upshape}([\cite{Mbekhta-Muller 9}])\end{upshape}
 If $T \in \mathcal{B}(X)$ and $Q \in \mathcal{B}(X)$ is a
 quasi-nilpotent operator commuting with $T$, then
\begin{equation}{\label{eq 3.1}}
 \qquad\qquad\qquad\quad \sigma(T+Q) = \sigma(T) \makebox{ and  }  \sigma_{a}(T+Q) = \sigma_{a}(T).
\end{equation}
\end{lemma}

As an immediate consequence of Theorem \ref{2.19} and Lemma
\ref{3.8}, we also obtain the following corollary which provide a
positive answer to Question \ref{3.7}.

\begin {corollary}${\label{3.9}}$
 Let $T \in \mathcal{B}(X)$ and let $N \in \mathcal{B}(X)$ be a
 nilpotent operator commuting with $T$. Then
 \begin{equation} {\label{eq 3.2}}
 \qquad\qquad\qquad\qquad\qquad\quad \ \    \Pi_{a}(T+N) = \Pi_{a}(T).
\end{equation}
\end{corollary}

Let $\Pi(T)$ denote the set of all poles of $T$. It is proved in
[\cite{Berkani-Zariouh partial}, Lemma 2.2] that if $T \in
\mathcal{B}(X)$ and $Q \in \mathcal{B}(X)$ is a nilpotent operator
commuting with $T$, then
 \begin{equation}{\label{eq 3.3}}
 \qquad\qquad\qquad\qquad\qquad\qquad \ \     \Pi(T+N) = \Pi(T).
\end{equation}

Let $E(T)$ and $E_{a}(T)$ denote the set of all isolated eigenvalues
of $T$ and the set of all eigenvalues of $T$ that are isolated in
$\sigma_{a}(T)$, respectively. That is, $$E(T) = \{\lambda \in
\makebox{iso}\sigma(T): 0 < \alpha(T- \lambda I)\}$$ and
$$E_{a}(T) =
\{\lambda \in \makebox{iso}\sigma_{a}(T): 0 < \alpha(T- \lambda
I)\}.$$ An operator $T \in \mathcal{B}(X)$ is called $a$-$isoloid$
if $\makebox{iso}\sigma_{a}(T)=E_{a}(T)$.

We set
 $\Pi^{0}(T)=\{\lambda \in \Pi(T): \alpha(T-\lambda I)<\infty\}$,
  $\Pi_{a}^{0}(T) =\{\lambda \in \Pi_{a}(T): \alpha(T-\lambda I) < \infty\}$,
   $E^{0}(T)=\{\lambda \in E(T): \alpha(T-\lambda I) < \infty\}$
   and
   $E_{a}^{0}(T)=\{\lambda \in E_{a}(T): \alpha(T-\lambda I) < \infty\}.$

Suppose that $T \in \mathcal{B}(X)$ and that $N \in \mathcal{B}(X)$
is a nilpotent operator commuting with $T$. Then from the proof of
[\cite{Berkani Zariouh}, Theorem 3.5], it follows that
$$0 < \alpha(T+N) \Longleftrightarrow 0 < \alpha(T)$$ and $$\alpha(T+N) < \infty  \Longleftrightarrow \alpha(T)  < \infty.$$
Hence by Equation (\ref{eq 3.1}), we have the following equations:
 \begin{equation} {\label{eq 3.4}}
 \qquad\qquad\qquad\qquad\qquad\qquad \ \     E(T+N) = E(T);
\end{equation} \vspace{-3mm}
\begin{equation} {\label{eq 3.5}}
 \qquad\qquad\qquad\qquad\qquad\qquad \ \     E_{a}(T+N) = E_{a}(T);
\end{equation} \vspace{-3mm}
\begin{equation} {\label{eq 3.6}}
 \qquad\qquad\qquad\qquad\qquad\qquad \ \     E^{0}(T+N) = E^{0}(T);
\end{equation} \vspace{-3mm}
\begin{equation} {\label{eq 3.7}}
 \qquad\qquad\qquad\qquad\qquad\qquad \ \     E_{a}^{0}(T+N) = E_{a}^{0}(T).
\end{equation}

An operator $T \in \mathcal{B}(X)$ is said to be $upper$
$semi$-$Weyl$ if $T$ is upper semi-Fredholm and ind$(T) \leq 0$. An
operator $T \in \mathcal{B}(X)$ is said to be $Weyl$ if $T$ is
Fredholm and ind$(T)=0$. For $T \in \mathcal{B}(X)$, let us define
the $upper$ $semi$-$Browder$ $spectrum$, the $Browder$ $spectrum$,
the $upper$ $semi$-$Weyl$ $spectrum$ and the $Weyl$ $spectrum$ of
$T$ as follows respectively:
$$ \sigma_{USB}(T) = \{ \lambda \in \mathbb{C}: T - \lambda I \makebox{ is not a upper
semi-Browder operator} \};$$
$$ \sigma_{B}(T) = \{ \lambda \in \mathbb{C}: T - \lambda I \makebox{ is not a Browder
operator} \};$$
$$ \sigma_{USW}(T) = \{ \lambda \in \mathbb{C}: T - \lambda I \makebox{ is not a upper
semi-Weyl operator} \};$$
$$ \sigma_{W}(T) = \{ \lambda \in \mathbb{C}: T - \lambda I \makebox{ is not a Weyl
operator} \}.$$

Suppose that $T \in \mathcal{B}(X)$ and that $R \in \mathcal{B}(X)$
is a Riesz operator commuting with $T$. Then it follows from
[\cite{Tylli}, Proposition 5] and [\cite{Rakocevic}, Theorem 1] that
 \begin{equation} {\label{eq 3.8}}
 \qquad\qquad\qquad\qquad\qquad\quad \ \     \sigma_{USW}(T+R) = \sigma_{USW}(T);
\end{equation} \vspace{-3mm}
\begin{equation} {\label{eq 3.9}}
 \qquad\qquad\qquad\qquad\qquad\qquad \ \     \sigma_{W}(T+R) =
 \sigma_{W}(T);
\end{equation}\vspace{-3mm}
\begin{equation} {\label{eq 3.10}}
 \qquad\qquad\qquad\qquad\qquad\quad \ \     \sigma_{USB}(T+R) = \sigma_{USB}(T);
\end{equation}\vspace{-3mm}
\begin{equation} {\label{eq 3.11}}
 \qquad\qquad\qquad\qquad\qquad\qquad \ \     \sigma_{B}(T+R) =
 \sigma_{B}(T).
\end{equation}

Suppose that $T \in \mathcal{B}(X)$ and that $Q \in \mathcal{B}(X)$
is a quasi-nilpotent operator commuting with $T$. Then, noting that
$\Pi^{0}(T)=  \sigma(T) \backslash \sigma_{B}(T)$ and
$\Pi_{a}^{0}(T)=\sigma_{a}(T) \backslash \sigma_{USB}(T)$ for any $T
\in \mathcal{B}(X)$, it follows from Equations (\ref{eq 3.1}),
(\ref{eq 3.10}) and (\ref{eq 3.11}) that
\begin{equation} {\label{eq 3.12}}
 \qquad\qquad\qquad\qquad\qquad\qquad \ \     \Pi^{0}(T+Q) =
 \Pi^{0}(T);
\end{equation} \vspace{-3mm}
\begin{equation} {\label{eq 3.13}}
 \qquad\qquad\qquad\qquad\qquad\qquad \ \     \Pi_{a}^{0}(T+Q) =
 \Pi_{a}^{0}(T).
\end{equation}

In the following table, we use the abbreviations $gaW$, $aW$, $gW$,
$W$, $(gw)$, $(w)$, $(gaw)$ and $(aw)$ to signify that an operator
$T \in \mathcal{B}(X)$ obeys generalized a-Weyl's theorem, a-Weyl's
theorem, generalized Weyl's theorem, Weyl's theorem, property
$(gw)$, property $(w)$, property $(gaw)$ and property $(aw)$. For
example, an operator $T \in \mathcal{B}(X)$ is said to obey
generalized a-Weyl's theorem (in symbol $T \in gaW$), if
$\sigma_{a}(T) \backslash \sigma_{USBW}(T) = E_{a}(T)$. Similarly,
the abbreviations $gaB$, $aB$, $gB$, $B$, $(gb)$, $(b)$, $(gab)$ and
$(ab)$ have analogous meaning with respect to Browder's theorem or
the properties.

\smallskip

\begin {center}

\begin{tabular} {|l|l|l|l|l} \hline

 $gaW$    &   $\sigma_{a}(T) \backslash \sigma_{USBW}(T) = E_{a}(T)$      &   $gaB$    &    $\sigma_{a}(T) \backslash \sigma_{USBW}(T) = \Pi_{a}(T)$        \\ \hline
 $aW$     &   $\sigma_{a}(T) \backslash \sigma_{USW}(T) = E_{a}^{0}(T)$   &   $aB$     &    $\sigma_{a}(T) \backslash \sigma_{USW}(T) = \Pi_{a}^{0}(T)$     \\ \hline
 $gW$     &   $\sigma(T) \backslash \sigma_{BW}(T) = E(T)$                &   $gB$     &    $\sigma(T) \backslash \sigma_{BW}(T) = \Pi(T)$                  \\ \hline
 $W$      &   $\sigma(T) \backslash \sigma_{W}(T) = E^{0}(T)$             &    $B$     &    $\sigma(T) \backslash \sigma_{W}(T) = \Pi^{0}(T)$               \\ \hline
 $(gw)$   &   $\sigma_{a}(T) \backslash \sigma_{USBW}(T) = E(T)$          &   $(gb)$   &    $\sigma_{a}(T) \backslash \sigma_{USBW}(T) = \Pi(T)$           \\ \hline
 $(w)$    &   $\sigma_{a}(T) \backslash \sigma_{USW}(T) = E^{0}(T)$       &    $(b)$   &    $\sigma_{a}(T) \backslash \sigma_{USW}(T) = \Pi^{0}(T)$        \\ \hline
 $(gaw)$  &   $\sigma(T) \backslash \sigma_{BW}(T) = E_{a}(T)$            &   $(gab)$  &    $\sigma(T) \backslash \sigma_{BW}(T) = \Pi_{a}(T)$             \\ \hline
 $(aw)$   &   $\sigma(T) \backslash \sigma_{W}(T) = E_{a}^{0}(T)$         &    $(ab)$  &     $\sigma(T) \backslash \sigma_{W}(T) = \Pi_{a}^{0}(T)$          \\ \hline

\end{tabular}

\end {center}

Weyl-Browder type theorems and properties, in their classical and
more recently in their generalized form, have been studied by a
large of authors. Theorem \ref{2.19} and Equations (\ref{eq
3.1})---(\ref{eq 3.13}) give us an unifying framework for
establishing commuting perturbational results of Weyl-Browder type
theorems and properties (generalized or not).

\begin {corollary}${\label{3.10}}$
$(1)$  If $T \in \mathcal{B}(X)$ obeys $gaW$
       \begin {upshape}(\end {upshape}resp. $aW,$ $gW,$ $W,$ $(gw),$ $(w)$,
       $(gaw),$ $(aw),$ $(gb),$  $(gab)$\begin {upshape})\end {upshape} and $N \in
       \mathcal{B}(X)$ is a nilpotent operator commuting with $T$, then
       $T+N$ also obeys $gaW$
       \begin {upshape}(\end {upshape}resp. $aW,$ $gW,$ $W,$ $(gw),$ $(w),$
      $(gaw),$ $(aw),$ $(gb),$  $(gab)$\begin {upshape})\end {upshape}.

$(2)$  If $T \in \mathcal{B}(X)$ obeys $gaB$
      \begin {upshape}(\end {upshape}resp. $aB,$ $gB,$ $B$\begin {upshape})\end {upshape} and $R \in
       \mathcal{B}(X)$ is a Riesz operator commuting with $T$, then $T+R$ also obeys
      $gaB$ \begin {upshape}(\end {upshape}resp. $aB,$ $gB,$ $B$\begin {upshape})\end {upshape}.

$(3)$  If $T \in \mathcal{B}(X)$ obeys $(b)$
      \begin {upshape}(\end {upshape}resp. $(ab)$\begin {upshape})\end {upshape} and $Q \in
       \mathcal{B}(X)$ is a quasi-nilpotent operator commuting with $T$, then $T+Q$ also obeys
       $(b)$ \begin {upshape}(\end {upshape}resp. $(ab)$\begin {upshape})\end {upshape}.

\end{corollary}

\begin{proof}
$(1)$ It follows directly from Theorem \ref{2.19} and Equations
(\ref{eq 3.1})---(\ref{eq 3.9}).

$(2)$ By [\cite{AmouchM ZguittiH equivalence}], we know that $T$
obeys $gB$ (resp. $gaB$) if and only if $T$ obeys $B$ (resp. $aB$)
for any $T \in \mathcal{B}(X)$. Note that $T$ obeys $B$ (resp. $aB$)
if and only if $\sigma_{W}(T)=\sigma_{B}(T)$ (resp.
$\sigma_{USW}(T)=\sigma_{USB}(T)$). Hence by Equations (\ref{eq
3.8})---(\ref{eq 3.11}), the conclusion follows immediately.

$(3)$ It follows directly from Equations (\ref{eq 3.1}), (\ref{eq
3.8}), (\ref{eq 3.9}), (\ref{eq 3.12}) and (\ref{eq 3.13}).
\end{proof}

The commuting perturbational results established in Corollary
\ref{3.10}, in particular, improve many recent results of
[\cite{Berkani-Amouch}, \cite{Berkani-Zariouh partial},
\cite{Berkani Zariouh}, \cite{Berkani Zariouh Functional Analysis},
\cite{Rashid gw}] by removing certain extra assumptions.

\begin {remark}${\label{3.11}}$ \begin{upshape}
 (1)  For generalized a-Weyl's theorem, part (1) of Corollary
     \ref{3.10} improves
     [\cite{Berkani Zariouh Functional Analysis}, Theorem 3.3] by removing the extra assumption that
     $E_{a}(T) \subseteq \makebox{iso}\sigma(T)$ and extends [\cite{Berkani Zariouh Functional Analysis}, Theorem 3.2].
     For property $(gw)$, on one hand, part (1) of Corollary
     \ref{3.10} improves [\cite{Rashid gw}, Theorem
     2.16] (resp. [\cite{Berkani-Amouch}, Theorem 3.6]) by removing the extra assumption that
     $T$ is a-isoloid (resp. $T$ is a-polaroid) and extends [\cite{Berkani Zariouh}, Theorem 3.8];
     on the other hand, our proof for it is a corrected proof of [\cite{Rashid gw}, Theorem
     2.16]. For property $(gab)$, part (1) of Corollary
    \ref{3.10} improves [\cite{Berkani Zariouh}, Theorem 3.2] by removing the extra assumption that
     $T$ is a-polaroid and extends [\cite{Berkani Zariouh}, Theorem
     3.4].

     For generalized Weyl's theorem (resp. property $(w)$, property $(gaw)$), part (1) of Corollary
     \ref{3.10} has been proved in [\cite{Berkani-Amouch}, Theorem 3.4]
     (resp. [\cite{Aiena-Biondi-Villafane}, Theorem 3.8] and [\cite{Berkani-Amouch}, Theorem 3.1], [\cite{Berkani Zariouh}, Theorem 3.6])
     by using a different method.

     For a-Weyl's theorem,
     some other commuting perturbational theorems for it have been proved in [\cite{Berkani Zariouh Functional Analysis}, \cite{Cao-Guo-Meng}, \cite{Oudghiri}].

     For Weyl's theorem (resp. property $(aw)$, property $(gb)$), part (1) of Corollary
     \ref{3.10} has been proved in [\cite{Oberai}, Theorem 3] (resp. [\cite{Berkani Zariouh}, Theorem
    3.5], [\cite{Zeng-Zhong}, Theorem 2.6]).

 (2) It has been discovered in [\cite{Aiena-Carpintero-Rosas}] that  Browder's theorem and a-Browder's theorem are stable under
    commuting Riesz perturbations.

 (3) For property $(b)$ (resp. $(ab)$), part (3) of Corollary
    \ref{3.10} extends [\cite{Berkani-Zariouh partial}, Theorem 2.1]
    (resp. [\cite{Berkani Zariouh}, Theorem 3.1]) from commuting nilpotent perturbations to commuting quasi-nilpotent perturbations.

  \end{upshape}

\end{remark}

We conclude this paper by some examples to illustrate our commuting
perturbational results of Weyl-Browder type theorems and properties
(generalized or not).

\bigskip

The following simple example shows that $gaW$, $aW,$ $gW,$ $W,$
$(gw),$ $(w)$, $(gaw)$ and $(aw)$ are not stable under commuting
quasi-nilpotent perturbations.

\begin {example}${\label{3.12}}$ \begin{upshape}
Let $Q: l_{2}(\mathbb{N}) \longrightarrow l_{2}(\mathbb{N})$ be a
quasi-nilpotent operator defined by
$$Q(x_{1},x_{2},\cdots )=(\frac{x_{2}}{2},\frac{x_{3}}{3}, \cdots ) \makebox{\ \ \  for all \ } (x_{n}) \in
l_{2}(\mathbb{N}).$$ Then $Q$ is quasi-nilpotent,
$\sigma(Q)=\sigma_{a}(Q)=\sigma_{W}(Q)=\sigma_{USW}(Q)=\sigma_{BW}(Q)=\sigma_{USBW}(Q)$
$=\{0\}$ and $E_{a}(Q)=E_{a}^{0}(Q)=E(Q)=E^{0}(Q)=\{0\}$. Take
$T=0.$ Clearly, $T$ satisfies $gaW$ (resp. $aW,$ $gW,$ $W,$ $(gw),$
$(w)$, $(gaw)$, $(aw)$), but $T+Q=Q$ fails $gaW$ (resp. $aW,$ $gW,$
$W,$ $(gw),$ $(w)$, $(gaw)$, $(aw)$).

  \end{upshape}

\end{example}

The following example was given in [\cite{Zeng-Zhong}, Example 2.14]
to show that property $(gb)$ is not stable under commuting
quasi-nilpotent perturbations. Now, we use it to illustrate that
property $(gab)$ is also unstable under commuting quasi-nilpotent
perturbations.

\begin {example}${\label{3.13}}$ \begin{upshape}
Let $U: l_{2}(\mathbb{N}) \longrightarrow l_{2}(\mathbb{N})$ be the
unilateral right shift operator defined by
$$U(x_{1},x_{2},\cdots )=(0,x_{1},x_{2}, \cdots ) \makebox{\ \ \  for all \ } (x_{n}) \in
l_{2}(\mathbb{N}).$$ Let $V: l_{2}(\mathbb{N}) \longrightarrow
l_{2}(\mathbb{N})$ be a quasi-nilpotent operator defined by
$$V(x_{1},x_{2},\cdots )=(0,x_{1},0,\frac{x_{3}}{3},\frac{x_{4}}{4} \cdots ) \makebox{\ \ \  for all \ } (x_{n}) \in
l_{2}(\mathbb{N}).$$ Let $N: l_{2}(\mathbb{N}) \longrightarrow
l_{2}(\mathbb{N})$ be a quasi-nilpotent operator defined by
$$N(x_{1},x_{2},\cdots )=(0,0,0,-\frac{x_{3}}{3},-\frac{x_{4}}{4} \cdots ) \makebox{\ \ \  for all \ } (x_{n}) \in
l_{2}(\mathbb{N}).$$ It is easy to verify that $VN=NV$. We consider
the operators $T$ and $Q$ defined by $T=U \oplus V$ and $Q=0 \oplus
N$, respectively. Then $Q$ is quasi-nilpotent and $TQ=QT$. Moreover,
$$\sigma(T) = \sigma(U) \cup \sigma(V) = \{\lambda \in
\mathbb{C}: 0 \leq |\lambda| \leq 1 \},$$
$$\sigma_{a}(T) = \sigma_{a}(U) \cup \sigma_{a}(V) = \{\lambda \in
\mathbb{C}: |\lambda| = 1 \} \cup \{0\},$$
$$\sigma(T+Q) = \sigma(U) \cup
\sigma(V+N) = \{\lambda \in \mathbb{C}: 0 \leq |\lambda| \leq 1 \}$$
and
$$\sigma_{a}(T+Q) = \sigma_{a}(U) \cup \sigma_{a}(V+N) = \{\lambda \in
\mathbb{C}: |\lambda| = 1 \} \cup \{0\}.$$ It follows that
$\Pi_{a}(T)=\Pi(T)= \varnothing$ and $\{0\}= \Pi_{a}(T+Q) \neq
\Pi(T+Q)= \varnothing.$ Hence by [\cite{Berkani-Zariouh New
Extended}, Corollary 2.7], $T+Q$ does not satisfy property $(gab)$.
But since $T$ has SVEP, $T$ satisfies Browder's theorem or
equivalently, by [\cite{AmouchM ZguittiH equivalence}, Theorem 2.2],
$T$ satisfies generalized Browder's theorem. Therefore by
[\cite{Berkani-Zariouh New Extended}, Corollary 2.7] again, $T$
satisfies property $(gb)$.
\end{upshape}

\end{example}

The following example was given in [\cite{Zeng-Zhong}, Example 2.12]
to show that property $(gb)$ is not preserved under commuting finite
rank perturbations. Now, we use it to illustrate that property $(b)$
and $(ab)$ are also unstable under commuting finite rank (hence
compact) perturbations.

\begin {example}${\label{3.14}}$ \begin{upshape}
Let $U: l_{2}(\mathbb{N}) \longrightarrow l_{2}(\mathbb{N})$ be the
unilateral right shift operator defined by
$$U(x_{1},x_{2},\cdots )=(0,x_{1},x_{2}, \cdots ) \makebox{\ \ \  for all \ } (x_{n}) \in
l_{2}(\mathbb{N}).$$ For fixed $0 < \varepsilon < 1$, let
$F_{\varepsilon}: l_{2}(\mathbb{N}) \longrightarrow
l_{2}(\mathbb{N})$ be a finite rank operator defined by
$$F_{\varepsilon}(x_{1},x_{2},\cdots )=(-\varepsilon x_{1},0,0, \cdots ) \makebox{\ \ \  for all \ } (x_{n}) \in
l_{2}(\mathbb{N}).$$ We consider the operators $T$ and $F$ defined
by $T=U \oplus I$ and $F=0 \oplus F_{\varepsilon}$, respectively.
Then $F$ is a finite rank operator and $TF=FT$. Moreover,
$$\sigma(T) = \sigma(U) \cup \sigma(I) = \{\lambda \in
\mathbb{C}: 0 \leq |\lambda| \leq 1 \},$$
$$\sigma_{a}(T) = \sigma_{a}(U) \cup \sigma_{a}(I) = \{\lambda \in
\mathbb{C}: |\lambda| = 1 \} ,$$
$$\sigma(T+F) = \sigma(U) \cup
\sigma(I+F_{\varepsilon}) = \{\lambda \in \mathbb{C}: 0 \leq
|\lambda| \leq 1 \}$$ and
$$\sigma_{a}(T+F) = \sigma_{a}(U) \cup \sigma_{a}(I+F_{\varepsilon}) = \{\lambda \in
\mathbb{C}: |\lambda| = 1 \} \cup \{ 1 - \varepsilon \}.$$ It
follows that $\Pi_{a}^{0}(T)=\Pi^{0}(T)= \varnothing$ and $\{1 -
\varepsilon \}= \Pi_{a}^{0}(T+F) \neq \Pi^{0}(T+F)= \varnothing.$
Hence by [\cite{Berkani-Zariouh Extended}, Corollary 2.7] (resp.
[\cite{Berkani-Zariouh New Extended}, Corollary 2.6]), $T+F$ does
not satisfy property $(b)$ (resp. $(ab)$). But since $T$ has SVEP,
$T$ satisfies a-Browder's theorem (resp. Browder's theorem),
therefore by [\cite{Berkani-Zariouh Extended}, Corollary 2.7] (resp.
[\cite{Berkani-Zariouh New Extended}, Corollary 2.6]) again, $T$
satisfies property $(b)$ (resp. $(ab)$).

  \end{upshape}

\end{example}



\end{document}